\documentclass[letterpaper, 10 pt, conference]{ieeeconf}  

\IEEEoverridecommandlockouts                              
\usepackage[utf8]{inputenc} 
\usepackage[T1]{fontenc}    
\usepackage{url}            
\usepackage{booktabs}       
\usepackage{amsfonts,amsmath,amssymb,mathrsfs}       
\usepackage{bbm}
\usepackage{balance}
\usepackage{mathtools,xparse}
\usepackage[makeroom]{cancel}
\usepackage[acronym]{glossaries}
\usepackage{color}  
\usepackage[ruled,vlined]{algorithm2e}
\usepackage{algpseudocode}
\usepackage{subcaption}
\usepackage{nicefrac}       
\usepackage{microtype}      
\usepackage{xcolor}         
\usepackage{graphicx}
\usepackage{cite}
\usepackage[prependcaption,colorinlistoftodos]{todonotes}
\usepackage[geometry]{ifsym}

\newcommand*{\QEDB}{\hfill\ensuremath{\square}}%

\newcommand{\prob}{\mathbb{P}}
\newcommand{\R}{\mathbb{R}}

\newcommand{\N}{\mathbb{N}}

\newcommand{\bbS}{\mathbb{S}}
\newcommand{\mc}[1]{\mathcal{#1}}

\newcommand{\col}{\mathrm{col}}

\newcommand{\bs}[1]{\boldsymbol{#1}}

\newcommand{\continuanceref}{}

\newtheorem{theorem}{Theorem}

\newtheorem{definition}{Definition}
\newtheorem{proposition}{Proposition}
\newtheorem{lemma}{Lemma}

\newtheorem{remark}{Remark}
\newtheorem{assumption}{Assumption}
\newtheorem{standing}{Standing Assumption}

\makeglossaries
\newacronym{}{}{}

\makeglossaries
\newacronym{NEP}{NEP}{Nash equilibrium problem}
\newacronym{GNEP}{GNEP}{generalized Nash equilibrium problem}
\newacronym{iid}{i.i.d.\@}{independent and identically distributed}
\newacronym{wrt}{w.r.t.\@}{with respect to}
\newacronym{psd}{psd}{positive semidefinite}
\newacronym{ls}{LS}{least squares}
\newacronym{gp}{GP}{Gaussian process}

\newglossaryentry{VI}
{
	name={VI},
	description={variational inequality},
	first={\glsentrydesc{VI} (\glsentrytext{VI})},
	plural={VIs},
	descriptionplural={variational inequalities},
	firstplural={\glsentrydescplural{VI} (VIs)}
}

\newglossaryentry{v-GNE}
{
	name={v-GNE},
	description={variational generalized Nash equilibrium},
	first={\glsentrydesc{v-GNE} (\glsentrytext{v-GNE})},
	plural={v-GNE},
	descriptionplural={variational generalized Nash equilibria},
	firstplural={\glsentrydescplural{v-GNE} (\glsentryplural{v-GNE})}
}

\newglossaryentry{GNE}
{
	name={GNE},
	description={generalized Nash equilibrium},
	first={\glsentrydesc{GNE} (\glsentrytext{GNE})},
	plural={GNE},
	descriptionplural={generalized Nash equilibria},
	firstplural={\glsentrydescplural{GNE} (\glsentryplural{GNE})}
}

\graphicspath{{figures/}}

\DeclareMathSizes{10}{10}{6}{4}

\title{\LARGE \bf
Learning equilibria with personalized incentives\\ in a class of nonmonotone games
}

\author{Filippo Fabiani, Andrea Simonetto and Paul J. Goulart 
	\thanks{F. Fabiani and P. J. Goulart are with the Department of Engineering Science, University of Oxford, OX1 3PJ, United Kingdom {\tt \footnotesize (\{filippo.fabiani, paul.goulart\}@eng.ox.ac.uk)}. A. Simonetto is with the UMA, {ENSTA Paris}, Institut Polytechnique de Paris, 91120 Palaiseau, France {\tt \footnotesize (andrea.simonetto@ensta-paris.fr)}. This work was partially supported through the Government’s modern industrial strategy by Innovate UK, part of UK Research and Innovation, under Project LEO (Ref. 104781).}%
}

\begin{document}

\maketitle
\thispagestyle{empty}
\pagestyle{empty}

\begin{abstract}
	We consider quadratic, nonmonotone generalized Nash equilibrium problems with symmetric interactions among the agents. Albeit this class of games is known to admit a potential function, its formal expression can be unavailable in several real-world applications. For this reason, we propose a two-layer Nash equilibrium seeking scheme in which a central coordinator exploits noisy feedback from the agents to design personalized incentives for them. By making use of those incentives, the agents compute a solution to an extended game, and then return feedback measures to the coordinator. We show that our algorithm returns an equilibrium if the coordinator is endowed with standard learning policies, and corroborate our results on a numerical instance of a hypomonotone game. 
\end{abstract}

\section{Introduction}
Several multi-agent applications are characterized by symmetric interactions, in terms of cost incurred by each entity for a given service, across each pair of agents modelled within a noncooperative game-theoretic framework. Under some fairness condition of the electricity market in smart grids and demand-side management \cite{zhu2011multi,cenedese2019charging}, for instance, the cost per energy unit varies in the same way for all the assets, according to the current demand of the market. Likewise, in traffic or congestion games agents usually experience costs that only depend on the number of users occupying shared resources \cite{neel2002game,altman2007evolutionary}. Recently, symmetric barrier functions are also used to enforce proximity-based constraints among agents in several coordination and formation control applications \cite{zhang2010cooperative,fabiani2018distributed}.

In this paper, we consider \glspl{GNEP}, a multi-agent modelling paradigm in which selfish decision-makers compete for shared resources, characterized by symmetric interactions among the agents.
The numerous benefit brought by such a symmetric structure to the \gls{GNEP} are indeed well-known \cite{la2016potential}. In particular, the presence of symmetries implies the existence of a \emph{potential function}, which is key for two main reasons: i) it implicitly guarantees the existence of a \gls{GNE} for the \gls{GNEP} at hand, which represents a desirable outcome of the game \cite{facchinei2007generalized}, and ii) in some cases it can be exploited directly in the design of Nash equilibrium seeking algorithms with convergence guarantees. This is a crucial feature in nonconvex/nonmonotone setting \cite{heikkinen2006potential,fabiani2019multi,cenedese2019charging}.

A mathematical expression of the underlying potential function is usually available only when the physics characterizing the \gls{GNEP} at hand is known, or the agents' cost functions are suitably designed to result in a potential game \cite{li2013designing}. In a general scenario, however, retrieving an expression of the potential function is a hard task \cite[Ch.~2]{la2016potential}, or in some cases it may simply be unknown due to, e.g., privacy reasons.

Thus, in the simplified setting of a quadratic, nonmonotone (specifically, hypomonotone \cite{gadjov2021exact}) \gls{GNEP} enjoying the existence of a potential function (\S \ref{sec:prob_def}), which is however assumed to be unavailable, we design a two-layer scheme that allows the agents to compute a \gls{GNE}. Specifically, in the outer loop we endow a central coordinator with an online learning procedure, e.g., \gls{ls} or \gls{gp}, aiming at iteratively exploiting the noisy agents' feedback to learn some of their private characteristic, namely the (pseudo-)gradient mappings obtained from the agents' cost functions. In the spirit of \cite{simonetto2019personalized,ospina2020personalized,notarnicola2020distributed}, such a reconstructed information is hence exploited to iteratively design \emph{personalized incentives} for the agents, which on their hand make use of those incentives to compute a \gls{GNE} in the inner loop, and finally return noisy feedback measures to the central coordinator. 	

In the proposed setting we leverage the symmetries characterizing the interactions among agents to design parametric personalized incentives (\S \ref{sec:algorithm}) that play a key role to prove the convergence of the algorithm. In particular, i) they serve as regularization terms for the cost functions of the agents, thus enabling them for the practical computation of a \gls{v-GNE} in the inner loop; and ii) they allow to achieve faster convergence rates through an appropriate tuning of few key parameters. 
As main results, we prove that the proposed two-layer iterative scheme converges to a \gls{GNE} by exploiting the consistency bounds typical of available learning procedures for the coordinator, as for example \gls{ls} or \gls{gp} (\S \ref{sec:stat_case}). 
We remark here that \gls{GNE} seeking in nonmonotone \gls{GNEP} is, in general, a hard task. Available solution algorithms are indeed either tailored for \glspl{NEP} with no coupling constraints, or involve generic \glspl{VI}, i.e., possibly not suitable for distributed computation \cite{konnov2006regularization,yin2011nash,konnov2014penalty,lucidi2020solving}. We conclude the paper by testing our findings on a numerical instance of the considered hypomonotone \gls{GNEP} (\S \ref{sec:num_sim}).

\subsubsection*{Notation}
$\bbS^{n}$ is the space of $n \times n$ symmetric matrices. For vectors $v_1,\dots,v_N\in\mathbb{R}^n$ and $\mc I=\{1,\dots,N \}$, we denote $\bs{v} \coloneqq (v_1 ^\top,\dots ,v_N^\top )^\top = \mathrm{col}((v_i)_{i\in\mc I})$ and $ \bs{v}_{-i} \coloneqq \col(( v_j )_{j\in\mc I\setminus \{i\}})$. With a slight abuse of notation, we also use $\bs{v} = (v_i,\bs{v}_{-i})$.
The mapping $F:\R^n \to \R^n$ is monotone on $\mc{X} \subseteq \R^n$ if $(F(x) - F(y))^\top(x - y) \, \geq  0$ for all $x, y \in \mc{X}$; strongly monotone if there exists a constant $c > 0$ such that $(F(x) - F(y))^\top(x - y) \geq c \|x - y\|^2$ for all $x, y \in \mc{X}$; hypomonotone if there exists a constant $c \geq 0$ such that $(F(x) - F(y))^\top(x - y) \geq -c \|x - y\|^2$ for all $x, y \in \mc{X}$.

\section{Mathematical problem setup}
\label{sec:prob_def}
We study a noncooperative game $\Gamma \coloneqq (\mc{I}, (\mc{X}_i)_{i \in \mc{I}}, (J_i)_{i \in \mc{I}})$ with $N$ agents, indexed by $\mc{I} \coloneqq \{1, \ldots, N\}$. Each agent $i \in \mc{I}$ controls variable $x_i$, constrained to a local set $\mc{X}_i \subseteq \R^{n_i}$, and aims at solving the following optimization problem:
\begin{equation}\label{eq:single_prob}
	\forall i \in \mc{I} : 
	\left\{
	\begin{aligned}
		&\underset{x_i \in \mc{X}_i}{\textrm{min}} & & J_i(x_i, \bs{x}_{-i}) \\
		&\hspace{.1cm}\textrm{ s.t. } & & A_i x_i + \sum_{j \in \mc{I} \setminus \{i\}} A_j x_j  \leq b,
	\end{aligned}	
	\right.
\end{equation}
where, for computational purposes, each function $J_i : \R^n \to \R$, $n \coloneqq \sum_{i \in \mc{I}} n_i$, has a quadratic (possibly aggregative) form 
$$
	J_i(x_i, \bs{x}_{-i}) \coloneqq \frac{1}{2} x_i^\top Q_i x_i + \left(\sum_{j\in\mc{I}\setminus\{i\}} C_{i,j} x_j + q_i \right)^\top x_i,
$$
for some $Q_i \in \bbS^{n_i}$, $C_{i,j} \in \R^{n_i \times n_j}$, and $q_i \in \R^{n_i}$. Thus, the collection of optimization problems in \eqref{eq:single_prob} amounts to a (quadratic) \gls{GNEP}, where $A_i \in \R^{m \times n_i}$ constraints each agent to share common resources $b \in \R^m$ with the other agents. Let us define the sets $\mc{X} \coloneqq \prod_{i \in \mc{I}} \mc{X}_i$, where each $\mc{X}_i$ is assumed to be nonempty, compact and convex, $\mc{X}_i(\bs{x}_{-i}) \coloneqq \{x_i \in \mc{X}_i \mid A_i x_i \leq b - \textstyle\sum_{j \in \mc{I} \setminus \{i\}} A_j x_j\}$, and the feasible set of $\Gamma$, $\Omega \coloneqq \{\bs{x} \in \mc{X} \mid A \bs{x} \leq b\}$, with $A \coloneqq [A_1 \ \cdots \ A_N] \in \R^{m \times n}$. We are interested in the following popular solution to a~\gls{GNEP}.

\begin{definition}\textup{(Generalized Nash equilibrium \cite{facchinei2007generalized})}\label{def:GNE}
	A strategy vector $\bs{x}^{\star} \in \Omega$ is a \gls{GNE} of the game $\Gamma$ if, for all $i \in \mc{I}$,
	$$
	J_i(x^{\star}_i, \bs{x}^{\star}_{-i}) \leq \underset{y_i \in \mc{X}_i(\bs{x}^{\star}_{-i})}{\textrm{inf}} \ J_i(y_i, \bs{x}^{\star}_{-i}).
	$$
	\QEDB
\end{definition}

Thus, $\bs{x}^{\star} \in \Omega$ is an equilibrium if no agent can decrease their cost by changing unilaterally $\bs{x}^{\star}_{-i}$ to any other feasible point. In the remainder, we make the following assumption.

\begin{standing}\label{standing:symmetry}
	For all $(i,j) \in \mc{I}^2$, $C_{i,j} = C_{j,i}$.~\QEDB
\end{standing}

Under the condition of symmetric interactions among agents, we know that the game mapping associated to $\Gamma$, $G: \R^n \to \R^n$, also known as the pseudo-gradient mapping since it is formally defined as $G(\bs{x}) \coloneqq \col((\nabla_{x_i} J_i(x_i, \bs{x}_{-i}))_{i \in \mc{I}})$, admits a differentiable, yet possibly \emph{unknown}, function $\theta : \R^n \to \R$ such that $G(\bs{x}) = \nabla \theta(\bs{x})$, for all $\bs{x} \in \Omega$ \cite[Th.~1.3.1]{facchinei2007finite}. This latter coincides with a potential function for the game $\Gamma$ \cite{facchinei2011decomposition}, for which we define the set of constrained minimizers $\Theta \coloneqq \textrm{argmin}_{\bs{y} \in \Omega} \, \theta(\bs{y})$, assumed to be nonempty, and set of constrained stationary points $\Theta^{\textrm{s}}$. Clearly, $\Theta \subseteq \Theta^{\textrm{s}}$. Note that, according to Definition~\ref{def:GNE}, any $\bs{x}^\star \in \Theta$ coincides with a \gls{GNE} for $\Gamma$, namely $\Theta \neq \emptyset$ guarantees the existence of at least a \gls{GNE} of the underlying game.
By considering quadratic costs in \eqref{eq:single_prob}, $G(\cdot)$ turns into an affine mapping:
$$
G(\bs{x}) =  \begin{bmatrix}
	Q_1 & \cdots & C_{1,N}\\
	\vdots & \ddots & \vdots\\
	C_{N,1} & \cdots & Q_N
\end{bmatrix} \bs{x} + \begin{bmatrix}
	q_1\\
	\vdots\\
	q_N
\end{bmatrix}  \eqqcolon Q \bs{x} +  q.
$$

We now introduce the function $\theta : \R^n \to \R$, which can be characterized as stated immediately below:
\begin{equation}\label{eq:pot_fun_hypo}
	\theta(\bs{x}) \coloneqq \sum_{i \in \mc{I}} \left( \frac{1}{2} x_i^\top Q_i x_i + q_i^\top x_i  + \textstyle\sum\limits_{j \in \mc{I}, j < i} (C_{i,j} x_j)^\top x_i \right).
\end{equation}
\begin{proposition}\label{prop:dis_fun_prop}
	For the function $\theta(\bs{x})$ in \eqref{eq:pot_fun_hypo} we have that:
	\begin{enumerate}
		\item[i)] It amounts to a potential function for the \gls{GNEP} $\Gamma$;
		\item[ii)] It is $\ell$-weakly convex, with $\ell \coloneqq |\lambda_{\textrm{min}}(Q)|$.
	\end{enumerate}
	\QEDB
\end{proposition}
\begin{proof}
	i) This part of the proof is akin to \cite[Prop.~2]{fabiani2019nash} and therefore, due to space limitations, it is here omitted. 
	
	ii) We start from the definition of hypomonotonicity for the mapping $\nabla \theta(\cdot)$, which requires the existence of some $\rho > 0$ such that $(\nabla \theta(\bs{x}) - \nabla \theta(\bs{y}))^\top (\bs{x} - \bs{y}) \geq -\rho \|\bs{x} - \bs{y}\|^2$, for all $\bs{x}$, $\bs{y} \in \Omega$. Since $\nabla \theta(\bs{x}) = G(\bs{x}) = Q \bs{x} + q $,
	$$
	\begin{aligned}
		(\nabla \theta(\bs{x}) - \nabla \theta(\bs{y}))^\top &(\bs{x} - \bs{y}) = (Q (\bs{x} - \bs{y}))^\top (\bs{x} - \bs{y})\\
		& = \|\bs{x} - \bs{y}\|^2_Q \geq \lambda_{\textrm{min}}(Q) \|\bs{x} - \bs{y}\|^2.
	\end{aligned}
	$$
	With $\lambda_{\textrm{min}}(Q)$ that, in principle, may be nonpositive (since each $Q_i$ is only symmetric), we obtain $(\nabla \theta(\bs{x}) - \nabla \theta(\bs{y}))^\top (\bs{x} - \bs{y}) \geq -|\lambda_{\textrm{min}}(Q)| \|\bs{x} - \bs{y}\|^2$. In view of \cite[Lemma~2.1]{davis2019stochastic}, the $|\lambda_{\textrm{min}}(Q)|$-hypomonotonicity of the mapping $\nabla \theta(\cdot)$ is equivalent to the $|\lambda_{\textrm{min}}(Q)|$-weak convexity of the function $\theta(\cdot)$, thus concluding the proof.
\end{proof}

It follows from Proposition~\ref{prop:dis_fun_prop} that the quadratic \gls{GNEP} $\Gamma$ in \eqref{eq:single_prob} turns into a hypomonotone \gls{GNEP} \cite{gadjov2021exact} in which the mapping $Q \bs{x} + q$ may not be monotone, since $Q + Q^\top$ is not positive semi-definite in general. As a consequence, the function $\theta(\cdot)$ may not be convex, thus posing several limitations in the design of an equilibrium seeking procedure for $\Gamma$ with most of the available techniques (monotonicity is, indeed, one of the weakest requirements \cite{facchinei2007finite}). 

On the other hand, we note that Standing Assumption~\ref{standing:symmetry} is key to claim that the underlying quadratic \gls{GNEP} admits a potential function. This latter is key for two main reasons: i) to claim that the game at hand possesses (at least) an equilibrium (in our case, $\Theta \neq \emptyset$), and ii) to design Nash equilibrium seeking algorithms with convergence guarantees, especially in nonconvex/nonmonotone setting, e.g., \cite{heikkinen2006potential,fabiani2019multi,cenedese2019charging}. 
In practical applications, however, may not be always reasonable to rely on a formal expression of the potential function \cite[Ch.~2]{la2016potential} due to privacy reasons for instance. In our  mathematical developments, we hence treat the \gls{GNEP} in \eqref{eq:single_prob} as if we were unaware of the fact that it is potential, and therefore the expression for $\theta(\cdot)$ in \eqref{eq:pot_fun_hypo} can not be directly exploited for algorithm design. By relying on standard learning paradigms, in the next section we design a two-layer procedure to steer the agents on some point falling into the set $\Theta$, which in view of Definition~\ref{def:GNE} coincides with a \gls{GNE} of $\Gamma$.

\section{The two-layer learning procedure}\label{sec:algorithm}
	\begin{algorithm}[!t]
	\caption{Two-layer semi-decentralized method}\label{alg:two_layer}
	\DontPrintSemicolon
	\SetArgSty{}
	\SetKwFor{ForAll}{for all}{do}{end forall}
	\smallskip
	\textbf{Initialization:} Set $t = 0$, $\bs{x}^\star_{t} \in \Omega$, set sequences $(c_t)_{t \in \N} > 2 \ell$ $(\xi_t)_{t \in \N} \in [0, 1/c_t)$, choose some $T \in \N$\\
	\smallskip
	\textbf{Iteration $(t \in \N)$:} \\
	\smallskip
	\begin{itemize}\setlength{\itemindent}{0cm}
		\item[$\bullet$] Integrate recent measures to learn $(\nabla_{x_i} \hat{J}_{i, t - 1}(\bs{x}^\star_{t -1}))_{i \in \mathcal{I}}$\\
		\smallskip
		\item[$\bullet$] Compute personalized incentives $(u_{i,t}(\bs{x}))_{i \in \mc{I}}$\\
		\smallskip
		\item[$\circ$] Compute a \gls{v-GNE} of the extended game $\overline{\Gamma}$, $\bs{x}^\star_t \in \Omega$\\
		\smallskip
		\item[$\bullet$] Obtain noisy agents' feedback $\{(x^\star_{i,t},p_{i,t})\}_{i \in \mc{I}}$
	\end{itemize}
\end{algorithm}

From the discussion in \S \ref{sec:prob_def}, we identify two critical challenges in designing a Nash equilibrium seeking method to compute a \gls{GNE} for $\Gamma$: 1) the non-monotonicity of $Q \bs{x} + q$, and 2) the unknown expression of the potential function $\theta$ in \eqref{eq:pot_fun_hypo}. We propose a way to circumvent both issues by designing \emph{personalized feedback functionals} $u_i : \R^{n} \to \R$ in the spirit of \cite{simonetto2019personalized}, which can used as ``control actions'' as described~next.

\subsection{The algorithm} 
 We summarize the key steps of the proposed, two-layer approach in Algorithm~\ref{alg:two_layer}. In particular, black-filled bullets refer to the tasks that have to be performed by a central coordinator, while the empty bullet to the one performed by the agents in $\mc{I}$. After the initialization phase, the coordinator aims at learning online (i.e., while the algorithm is running) the gradient mapping of the unknown function $\theta$ by leveraging possibly noisy agents' feedback on the private functions $J_i$'s in the outer loop. Then, on the basis of the estimated $\hat{J}_{i,t}$, the central entity designs personalized incentive functionals $u_{i,t}$, thus forcing the agents to face with an extended version of the quadratic \gls{GNEP} $\Gamma$ in \eqref{eq:single_prob}, i.e., $\overline{\Gamma} \coloneqq (\mc{I}, (\mc{X}_i)_{i \in \mc{I}}, (\bar{J}_{i,t})_{i \in \mc{I}})$, which stem from replacing each cost function $J_i(x_i,  \bs{x}_{-i})$ in \eqref{eq:single_prob} with $J_i(x_i,  \bs{x}_{-i}) + u_{i,t}(x_i,  \bs{x}_{-i}) \eqqcolon \bar{J}_{i,t}(x_i,  \bs{x}_{-i})$. 

By suitably choosing such incentives, we will show that they serve as regularization terms for $J_i(\cdot,  \bs{x}_{-i})$, as well as they trade-off convergence properties of Algorithm~\ref{alg:two_layer} and robustness to the imperfect knowledge of the potential function $\theta$ and associated gradient. Specifically, personalized incentives enable the agents for the practical computation of an equilibrium of the extended game $\overline{\Gamma}$ via standard solution procedures for \glspl{GNEP} \cite{salehisadaghiani2016distributed,ye2017distributed,gadjov2021exact}. In the cognate literature, however, algorithmic methods typically return a \gls{v-GNE} \cite{facchinei2007generalized}, which coincides to any solution to the (extended) \gls{GNEP} $\overline{\Gamma}$ that is also a solution to the associated \gls{VI}, i.e., any collective vector of strategies $\bs{x}_{t}^\star \in \Omega$ such that, for all $t \in \N$,	
\begin{equation}\label{eq:VI}
	(\bs{y} - \bs{x}_{t}^\star)^\top \bar{G}_{t}(\bs{x}_{t}^\star) \geq 0, \, \text{ for all } \bs{y} \in \Omega,
\end{equation}	
where the mapping $\bar{G}_{t} : \R^n \to \R^n$ is formally obtained as $\bar{G}_t(\bs{x}) \coloneqq \col((\nabla_{x_i} \bar{J}_{i,t}(x_i, \bs{x}_{-i}))_{i \in \mc{I}}) = Q \bs{x} + q + U_t(\bs{x})$, and $U_t(\bs{x}) \coloneqq \col((\nabla_{x_i} u_{i,t}(x_i, \bs{x}_{-i}))_{i \in \mc{I}})$. 
This is why the computational step in Algorithm~\ref{alg:two_layer} involving the agents requires them to implement an available procedure to compute a \gls{v-GNE} of the extended game $\overline{\Gamma}$.
As a last step, the proposed method requires the coordinator to retrieve feedback measures and equilibrium strategies from the agents, $\{(x^\star_{i,t},p_{i,t})\}_{i \in \mc{I}}$, with $p_{i,t} \coloneqq J_i(\bs{x}^\star_t) + \varepsilon_{i, t}$ and random variable $\varepsilon_i$.

\subsection{Personalized incentives design}		
We follow the approach in \cite[Alg.~1]{ospina2020personalized} by endowing the central coordinator with a \emph{learning procedure} $\mathscr{L}$ such that, at every outer iteration $t \in \N$ of Algorithm~\ref{alg:two_layer}, it integrates the most recent noisy agents' feedback $\{p_{i,t-1}\}_{i \in \mc{I}}$ to calculate an estimate of the gradients $(\nabla_{x_i} \hat{J}_{i, t-1}(\bs{x}^\star_{t-1}))_{i \in \mc{I}}$ (since $G(\bs{x}) = \nabla \theta(\bs{x})$ \cite[Th.~1.3.1]{facchinei2007finite}). 
Then, it exploits such estimates to design personalized functionals as follows:
\begin{equation}\label{eq:pers_feedback}
	u_{i,t}(\bs{x}) = \tfrac{1}{2} c_t \|x_i - x^+_{i,t}\|^2, \text{ for all } i \in \mc{I},
\end{equation}	
where $x^+_{i,t} \coloneqq x^\star_{i, t-1} + \xi_t \nabla_{x_i} \hat{J}_{i, t-1}(\bs{x}^\star_{t-1})$, for some $c_t$, $\xi_t  \geq 0$, for all $t \in \N$. 
Given the parametric form in \eqref{eq:pers_feedback}, our goal hence reduces to design suitable conditions for the gain $c_t$ and step-size $\xi_t$ in such a way that the sequence of \gls{v-GNE}, $(\bs{x}^\star_t)_{t \in \N}$, monotonically decreases (or non-increases) the unknown potential function $\theta(\cdot)$, and asymptotically converges to one of its constrained minimizers (either local or global). 

The structure of the proposed personalized incentives brings us to the following considerations: i) inspired by recent algorithms based on the Heavy Anchor method, e.g.,  \cite{gadjov2021exact}, we note that each term $x^+_{i,t}$ requires a positive sign for the gradient step $\xi_t \nabla_{x_i} \hat{J}_{i, t-1}(\bs{x}^\star_{t-1})$. In the next sections we will show that such a choice enables us to boost the convergence of Algorithm~\ref{alg:two_layer}, or to reduce the asymptotic error by means of a suitable tuning of $\xi_t$. In addition, ii) we note that $c_t$ is crucial to allow the agents for the practical computation of a \gls{v-GNE} through available iterative schemes, as stated next:
\begin{proposition}\label{prop:strong_conv}
	Let $c_t \geq 2 \ell$ for all $t \in \N$. Then, with the personalized incentives in \eqref{eq:pers_feedback}, the mapping $\bar{G}_t(\cdot)$ is $\ell$-strongly monotone, for all $t \in \N$.
	\QEDB
\end{proposition}
\begin{proof}
	By using the definition of incentive functionals in \eqref{eq:pers_feedback},  we obtain $\bar{G}_t(\bs{x}) = G(\bs{x}) + c_t (\bs{x} - \bs{x}^\star_{t-1} - \xi_t \hat{G}_{t-1}(\bs{x}^\star_{t-1}))$. Thus, for any $\bs{x}$, $\bs{y} \in \Omega$, and $t \in \N$, we have that:
	$$
	\begin{aligned}
		(\bs{x} &- \bs{y})^\top (\bar{G}_t(\bs{x}) - \bar{G}_t(\bs{y}))\\
		& = (\bs{x} \!-\! \bs{y})^\top (G(\bs{x}) \!+\! c_t (\bs{x} \!-\! \bs{x}^\star_{t-1} \!-\! \xi_t\hat{G}_{t-1}(\bs{x}^\star_{t-1})) \\
		& \hspace{2.5cm} \!-\! G(\bs{y}) \!-\! c_t (\bs{y} \!-\! \bs{x}^\star_{t-1} \!-\! \xi_t\hat{G}_{t-1}(\bs{x}^\star_{t-1})))\\
		& = (\bs{x} \!-\! \bs{y})^\top (G(\bs{x}) \!+\! c_t \bs{x} \!-\! G(\bs{y}) \!-\! c_t \bs{y})\\
		& = (\bs{x} \!-\! \bs{y})^\top (\nabla\theta(\bs{x}) \!+\! c_t \bs{x} \!-\! \nabla\theta(\bs{y}) \!-\! c_t \bs{y}),
	\end{aligned}
	$$
	where the last equality follows from Standing Assumption~\ref{standing:symmetry} and \cite[Th.~1.3.1]{facchinei2007finite}. Let us now consider the auxiliary function $\psi_t(\bs{x}) \coloneqq \theta(\bs{x}) + \tfrac{c_t}{2} \| \bs{x} \|^2$. In case $c_t \geq 2 \ell$ for all $t \in \N$, $\psi_t(\cdot)$ turns out to be $\ell$-strongly convex in view of the $\ell$-weak convexity of $\theta$ (Proposition~\ref{prop:dis_fun_prop}), which directly leads to
	$$
	\begin{aligned}
		&(\bs{x} - \bs{y})^\top (\bar{G}_t(\bs{x}) - \bar{G}_t(\bs{y}))\\ 
		&\hspace{1.8cm} =  (\bs{x} - \bs{y})^\top (\nabla\theta(\bs{x}) + c_t \bs{x} - \nabla\theta(\bs{y}) - c_t \bs{y})\\
		&\hspace{1.8cm} = (\bs{x} - \bs{y})^\top (\nabla\psi_t(\bs{x}) - \nabla\psi_t(\bs{y})) \geq \ell \|\bs{x} - \bs{y}\|^2,
	\end{aligned}
	$$
	i.e., the definition of strongly monotone mapping for $\bar{G}_t$.
\end{proof}

It follows that, if $c_t$ is large enough for all $t \in \N$, then at every outer iteration the agents compute the unique  \cite[Th.~2.3.3]{facchinei2007finite} \gls{v-GNE} associated to the extended \gls{GNEP} $\overline{\Gamma}$.

\section{Main results}\label{sec:stat_case}	
We now establish convergence results for Algorithm~\ref{alg:two_layer} by distinguishing between two cases: online perfect reconstruction, if the learning procedure $\mathscr{L}$ allows the central coordinator to leverage $\nabla_{x_i} \hat{J}_{i, {t-1}}(\bs{x}^\star_{t-1}) = \nabla_{x_i} J_{i}(\bs{x}^\star_{t-1})$, for all $i \in \mc{I}$ and $t \in \N$, and imperfect reconstruction. 
To this end, we introduce the fixed point residual $\Delta^\star_t \coloneqq \bs{x}^\star_{t} - \bs{x}^\star_{t-1}$ as a key quantity to ``measure'' the distance between the strategy profile at the current iteration and the points in $\Theta^{\textrm{s}}$.
\begin{lemma}\label{lemma:stationary}
	Let $(\bs{x}^\star_t)_{t \in \N}$ be the sequence of \gls{v-GNE} generated by Algorithm~\ref{alg:two_layer} with $c_t \xi_t < 1$. Moreover, assume perfect reconstruction of the mapping $Q \bs{x} + q$, and that there exists some $\bs{x}^\star_t \in \Omega$ such that $\|\Delta_t^\star\| = 0$. Then, some $\bs{x}^\star_t \in \Omega$, i.e., $\bs{x}^\star_t$ is a stationary point for the function $\theta(\bs{x})$.
	\QEDB
\end{lemma}	
\begin{proof}
	Any stationary point of the function $\theta(\bs{x})$ satisfies the following first-order optimality condition: $$w (Q \bs{x} + q) + \mathcal{N}_{\Omega}(\bs{x}) \ni \bs{0},$$ where $w>0$ scales the cost associated to $\theta(\cdot)$, while $ \mathcal{N}_{\Omega}$ denotes the normal cone operator of the feasible set $\Omega$. At every $t \in \N$, the \gls{v-GNE} $\bs{x}^\star_t$ obtained with perfect reconstruction satisfies
	$
	(Q \bs{x}^\star_t + q) + c_t (\bs{x}^\star_t - \bs{x}^\star_{t-1} - \xi_t (Q \bs{x}^\star_{t-1} + q)) + \mathcal{N}_{\Omega}(\bs{x}^\star_t) \ni \bs{0}.
	$
	If $\|\Delta^\star_t\| = 0$ then we have $\bs{x}^\star_t = \bs{x}^\star_{t-1}$, namely $\bs{x}^\star_t$ is the solution to the inclusion
	$$
	(1- c_t\xi_t) (Q \bs{x}^\star_{t} + q) + \mathcal{N}_{\Omega}(\bs{x}^\star_t) \ni \bs{0},
	$$
	thus satisfying the first-order optimality conditions for $\theta(\bs{x})$, and therefore it amounts to one of its stationary points. 
\end{proof}

In case of inexact reconstruction of the gradient mappings, as a metric for the convergence of the sequence $(\bs{x}^\star_t)_{t \in \N}$ to the stationary point set we also adopt the average value of $ \|\Delta_t^\star\|$ over a certain iterations horizon of length $T \in \N$, formally defined as $\frac{1}{T}\sum_{t\in \mc{T}} \|\Delta^\star_t\|$, $\mc{T}\coloneqq\{1, \ldots, T\}$. For the imperfect reconstruction case, our result will hence be of the form  $(1/T)\sum_{t\in \mc{T}} \|\Delta^\star_t\| = O(1)$ if the reconstruction error is persistent, while $(1/T)\sum_{t\in \mc{T}} \|\Delta^\star_t\| = 0$ otherwise.
\begin{remark}
	Since our algorithm produces monononically non-increasing values for $\theta$ through the iterations, it is worth mentioning that the application of traditional perturbation techniques, such as in \cite{escape}, can ensure that the stationary points to which we converge are, in practice, local minima in $\Theta^{\textrm{s}}$, and hence \gls{GNE} of the original game $\Gamma$.
	\QEDB
\end{remark}   

\subsection{Perfect reconstruction of the pseudo-gradients}\label{sec:stat_case_1}

In the desirable, yet unrealistic, case in which the procedure $\mathscr{L}$ enables for the perfect reconstruction, i.e., $\nabla_{x_i} \hat{J}_{i, {t-1}}(\bs{x}^\star_{t-1}) = \nabla_{x_i} J_{i}(\bs{x}^\star_{t-1})$, for all $i \in \mc{I}$ and $t \in \N$, we show that a careful choice of the gain $c_t $ and the step-size $\xi_t $ characterizing the personalized incentives in \eqref{eq:pers_feedback} allows Algorithm~\ref{alg:two_layer} to produce a convergent sequence of \gls{v-GNE}:
\begin{proposition}\label{prop:convergence_perfect}
	For all $t \in \N$, let $c_t \geq 2 \ell$ and $\xi_t \in [0, 1/c_t)$. Then, with the incentives in \eqref{eq:pers_feedback} the sequence of \gls{v-GNE} $(\bs{x}^\star_t)_{t \in \N}$ generated by Algorithm~\ref{alg:two_layer} has limit point in $\Theta^{\textrm{s}}$.
	\QEDB
\end{proposition}
\begin{proof}
	We first show that, if $c_t$ (resp., $\xi_t$) is large (small) enough, the personalized functionals in \eqref{eq:pers_feedback} allow to point a descent direction for the unknown potential function $\theta$, i.e., $\Delta^{\star^\top}_t \nabla \theta(\bs{x}^\star_{t-1})  < 0$.
	Since $\bs{x}^\star_t$ amounts to a \gls{v-GNE} at every outer iteration $t \in \N$, we have by \eqref{eq:VI} that $(\bs{y} - \bs{x}^\star_t)^\top \bar{G}_t(\bs{x}^\star_t) \geq 0$ for all $\bs{y} \in \Omega$. Then, since $\bs{x}^\star_{t-1} \in \Omega$ as it is a \gls{v-GNE} at $t - 1$, we also have $\Delta_t^{\star^\top} \bar{G}_t(\bs{x}^\star_t) \leq 0$. By adding and subtracting the term $\Delta_t^{\star^\top} G(\bs{x}^\star_{t-1})$, we hence obtain	
	$$
	\Delta_t^{\star^\top} G(\bs{x}^\star_{t-1}) \leq \Delta_t^{\star^\top} (G(\bs{x}^\star_{t-1}) - (G(\bs{x}^\star_t) + U_t(\bs{x}^\star_t))).
	$$
	In addition, by combining the definition of the incentives in \eqref{eq:pers_feedback} with the perfect reconstruction of the pseudo-gradients, we also have that $U_t(\bs{x}^\star_t) = c_t (\Delta_t^\star - \xi_t G(\bs{x}^\star_{t-1}))$.
	Then, in view of the symmetry postulated in Standing Assumption~\ref{standing:symmetry}, it follows that 
	$$
	\begin{aligned}
		(1 - c_t \xi_t) &\Delta_t^{\star^\top} \nabla \theta(\bs{x}^\star_{t-1}) \\
		&\leq \Delta_t^{\star^\top} ( \nabla \theta(\bs{x}^\star_{t-1}) + c_t \bs{x}^\star_{t-1}  - \nabla \theta(\bs{x}^\star_t) - c_t \bs{x}^\star_{t}).
	\end{aligned}
	$$
	We now define $\alpha_t \coloneqq 1 - c_t \xi_t$ for all $t \in \N$, and introduce the auxiliary function $\psi_t(\bs{x}) \coloneqq \theta(\bs{x}) + \tfrac{c_t}{2} \| \bs{x} \|^2$, which is known to be $\ell$-strongly convex in case $c_t \geq 2 \ell$ (Proposition~\ref{prop:dis_fun_prop}). For this reason, we have $\Delta_t^{\star^\top} ( -\nabla \psi_t(\bs{x}^\star_{t-1}) + \nabla \psi_t(\bs{x}^\star_t)) \geq \ell \| \Delta_t^\star \|^2$, and hence
	$$
	\alpha_t \Delta_t^{\star^\top} \nabla \theta(\bs{x}^\star_{t-1}) \! \leq \! \Delta_t^{\star^\top} ( \nabla \psi_t(\bs{x}^\star_{t-1}) - \nabla \psi_t(\bs{x}^\star_t)) \! \leq \! -\ell \| \Delta_t^\star \|^2.
	$$
	By imposing $\alpha_t > 0$, yielding to $\xi_t < 1/c_t \leq 1/2\ell$, for all $t \in \N$, we finally obtain $\Delta^{\star^\top}_t \nabla \theta(\bs{x}^\star_{t-1}) \leq - (\ell/\alpha_t) \| \Delta_t^\star \|^2 < 0$.
	This result can be then combined with the descent lemma \cite[Prop.~A.24]{bertsekas1997nonlinear} to claim that the sequence $(\bs{x}^\star_t)_{t \in \N}$ satisfies
	\begin{equation}\label{eq:descent_dir}
		\begin{aligned}
			\theta(\bs{x}^\star_t) &\leq \theta(\bs{x}^\star_{t - 1}) + (\Delta_t^\star)^\top \nabla \theta(\bs{x}^\star_{t - 1}) + \tfrac{\ell}{2} \|\Delta_t^\star\|^2\\
			&\leq \theta(\bs{x}^\star_{t - 1}) - \ell \tfrac{2 - \alpha_t}{2 \alpha_t} \|\Delta_t^\star\|^2.
		\end{aligned}
	\end{equation}
	Thus, by imposing $(2 - \alpha_t)/2 \alpha_t > 0$, which implies that $0 \leq \xi_t < 1/c_t \leq 1/2\ell$, 
	it must happen that the sequence $(\theta(\bs{x}^\star_t))_{t \in \N}$ tends to a finite value, as $\theta(\bs{x}^\star_t) \to -\infty$ can not happen in view of the boundedness of $\Omega$. Since $\theta$ is continuous, the convergence of $(\theta(\bs{x}^\star_t))_{t \in \N}$ implies that $\mathrm{lim}_{t \to \infty} \, \|\Delta_t^\star\| = 0$, and hence the bounded sequence of feasible points $(\bs{x}^\star_t)_{t \in \N} \in \Omega$ (since any $\bs{x}^\star_t$ is a \gls{v-GNE} of $\overline{\Gamma}$) has a limit point in $\Theta^{\textrm{s}}$ in view of of Lemma~\ref{lemma:stationary}.
\end{proof}

 From \eqref{eq:descent_dir} it is clear that making the product $c_t\xi_t$ close to one allows us to actually boost the convergence of Algorithm~\ref{alg:two_layer} to some point in $\Theta^{\textrm{s}}$. This indeed supports the choice for a positive sign in the gradient step of \eqref{eq:pers_feedback}.
In view of the presence of noise in the agents' feedback $\{p_{i,t-1}\}_{i \in \mc{I}}$, however, at least at the beginning of the iterative scheme in Algorithm~\ref{alg:two_layer} it seems unlikely that the learning procedure $\mathscr{L}$ guarantees a perfect reconstruction of $(\nabla_{x_i} \hat{J}_{i, {t-1}}(\bs{x}^\star_{t-1}))_{i \in \mc{I}}$. For this reason, we investigate next the effect of the inexact reconstruction on the convergence properties of Algorithm~\ref{alg:two_layer}.

\subsection{Inexact estimate of the pseudo-gradients}\label{sec:stat_case_2}
In this section we leverage a further assumption characterizing the reconstruction of the pseudo-gradient mappings. Inspired by \cite{dixit2019online,ospina2020personalized}, however, we make the following assumption on $\hat{G}_{t}(\cdot)$ directly, rather than on each gradient.
\begin{assumption}\label{ass:recons_error}
	For all $t \in \N$ and $\bs{x} \in \mc{X}$, $\hat{G}_t(\bs{x}) \coloneqq Q \bs{x} + q + \epsilon_t$, and, for any $\delta \in (0,1]$, there exists $\bar{t} < \infty$ such that $\prob\{\|\epsilon_{t}\| \leq \mathrm{e}(\bar{t}) \mid \forall t \geq \bar{t}\} \geq 1-\delta$, for some nonincreasing function $\mathrm{e} : \N \to \R_{\geq 0}$ such that $\mathrm{e}(t) < \infty$, for all $t \in \N$.
	\QEDB
\end{assumption}

With Assumption~\ref{ass:recons_error} we essentially require that, for all $\bs{x} \in \mc{X}$, the reconstruction error on $Q \bs{x} + q$ made by the learning procedure $\mathscr{L}$ is bounded with (possibly high) probability $1-\delta$ by some function of the available agents' feedback. Then, by introducing quantities $\kappa_t \coloneqq (1-\alpha_t)/2\alpha_t$ and $\beta_t \coloneqq \ell (2 - \alpha_t)/2 \alpha_t$, we can prove the following result.	
\begin{lemma}\label{lemma:desc_perfect_recon}
	Let $c_t \geq 2 \ell$ and $\xi_t \in [0, 1/c_t)$ for all $t \in \N$, and let Assumption~\ref{ass:recons_error} be valid for some $\delta \in (0,1]$. Then, the sequence of \gls{v-GNE} $(\bs{x}^\star_t)_{t \in \N}$ generated by Algorithm~\ref{alg:two_layer} with personalized incentives in \eqref{eq:pers_feedback} satisfies, with probability $1 - \delta$,
	\begin{equation}\label{eq:descent_ball}
		\begin{aligned}
			\Delta^{\star^\top}_t \nabla \theta(\bs{x}^\star_{t-1}) \leq& \tfrac{\alpha_t \kappa^2_t}{\ell}  \mathrm{e}^2(t-1) \\
			&\hfill- \left(\sqrt{\tfrac{\ell}{\alpha_t}} \, \| \Delta_t^\star \| - \kappa_t \sqrt{\tfrac{\alpha_t}{\ell}} \, \mathrm{e}(t-1) \right)^2.
		\end{aligned}
	\end{equation}
	 \QEDB
\end{lemma}
\begin{proof}
	We start by making use of the same steps performed at the beginning of the proof of Proposition~\ref{prop:convergence_perfect}, which in view of Assumption~\ref{ass:recons_error} leads to
	$
	\alpha_t \Delta_t^{\star^\top} \nabla \theta(\bs{x}^\star_{t-1}) \leq \Delta_t^{\star^\top}  ( \nabla \theta(\bs{x}^\star_{t-1}) + c_t \bs{x}^\star_{t-1}  - \nabla \theta(\bs{x}^\star_t) - c_t \bs{x}^\star_{t}) + c_t \xi_t \Delta_t^{\star^\top}  \epsilon_{t-1},
	$
	for all $t \geq \bar{t}$. Therefore, if $c_t \geq 2 \ell > 0$ for all $t \in \N$, we can upper bound $\Delta_t^{\star^\top} ( \nabla \theta(\bs{x}^\star_{t-1}) + c_t \bs{x}^\star_{t-1}  - \nabla \theta(\bs{x}^\star_t) - c_t \bs{x}^\star_{t}) $ with $ -\ell \| \Delta_t^\star \|^2$ due to the $\ell$-strong convexity of the auxiliary function $\psi_t(\cdot)$. Since $\xi_t \geq 0$, $\Delta_t^{\star^\top}  \epsilon_{t-1}$ attains its maximum positive module when $\Delta_t^\star$ and $\epsilon_{t-1}$ are aligned, and hence we obtain the following chain of inequalities
	$$
	\begin{aligned}
		\alpha_t \Delta_t^{\star^\top} \nabla \theta(\bs{x}^\star_{t-1}) &\leq  -\ell \| \Delta_t^\star \|^2 + c_t \xi_t \|\Delta_t^\star\|\| \epsilon_{t-1}\|\\
		&\leq -\ell \| \Delta_t^\star \|^2 + c_t \xi_t \|\Delta_t^\star\| \mathrm{e}(t-1).
	\end{aligned}
	$$
	Note that the last inequality holds with probability $1 - \delta$, as it follows directly from Assumption~\ref{ass:recons_error}, for any $t \geq \bar{t}$.
	Thus, for $\alpha_t > 0$, i.e., $\xi_t < 1/c_t$, and $c_t \xi_t = 1 - \alpha_t$, we obtain
	\begin{equation}\label{eq:descent}
		\Delta_t^{\star^\top} \nabla \theta(\bs{x}^\star_{t-1}) \!\leq\!  - \tfrac{\ell}{\alpha_t} \| \Delta_t^\star \|^2 + \tfrac{1-\alpha_t}{\alpha_t} \| \Delta_t^\star \| \mathrm{e}(t-1).
	\end{equation}
	Finally, completing the square in the RHS of \eqref{eq:descent} by adding and subtracting $((1-\alpha_t)^2/4\alpha_t\ell) \mathrm{e}^2(t-1)$  leads to
	$$
	\begin{aligned}
		\Delta_t^{\star^\top} \nabla \theta(\bs{x}^\star_{t-1}) \leq &-\!\left(\sqrt{\tfrac{\ell}{\alpha_t}} \, \| \Delta_t^\star \| \!-\! (1-\tfrac{\alpha_t}{2 \alpha_t}) \sqrt{\tfrac{\alpha_t}{\ell}} \, \mathrm{e}(t-1) \right)^2\\
		&\hspace{3cm} + \tfrac{(1-\alpha_t)^2}{4\alpha_t\ell}  \mathrm{e}^2(t-1).
	\end{aligned}
	$$
	The proof ends by substituting $\kappa_t$ in the inequality above.
\end{proof}

From the relation in \eqref{eq:descent_ball}, we note that the vector $\Delta_t^\star$ does not necessarily point a descent direction for the function $\theta$.
The term $\mathrm{e}^2(t-1)$, in fact, prevents the LHS in \eqref{eq:descent_ball} from being strictly negative, though it can be made small through $\kappa_t$ by means of a suitable choice of the step-size $\xi_t$. By introducing the quantities $\bar{\beta} \coloneqq \sum_{t \in \mc{T}} \ \beta_t$ and $\underline{\beta} \coloneqq \textrm{min}_{t \in \mc{T}} \ \beta_t$ over some horizon $\mc{T}$ of length $T \geq 0$, the following result characterizes the sequence of \gls{v-GNE} generated by Algorithm~\ref{alg:two_layer}.	
\begin{theorem}\label{th:conv_recons}
	Let $c_t\geq 2 \ell$ and $\xi_t \in [0, 1/c_t)$, for all $t \in \N$, and let Assumption~\ref{ass:recons_error} be valid for some $\delta \in (0,1]$. Let some $T \in \N$ be fixed, $\mc{T} \coloneqq \{\bar{t} + 1, \ldots, T + \bar{t}\}$ and, for any global minimizer $\bs{x}^\star \in \Theta$, $\Delta_{\bar{t}} \coloneqq \theta(\bs{x}^\star_{\bar{t}}) - \theta(\bs{x}^\star)$. The sequence of \gls{v-GNE} $(\bs{x}^\star_t)_{t \in \mc{T}}$ generated by Algorithm~\ref{alg:two_layer} with the personalized incentives in \eqref{eq:pers_feedback} satisfies, with probability~$1-\delta$,
	\begin{equation}\label{eq:seq_recons}
		\begin{aligned}
			\tfrac{1}{T} \sum_{t\in \mc{T}} \|	\Delta_t^\star	\| &\leq \tfrac{1}{T \underline{\beta}} \sqrt{\sum_{t \in \mc{T}} \left( \beta_t \Delta_{\bar{t}} + \tfrac{\bar{\beta} \kappa^2_t}{\beta_t} \mathrm{e}^2(t) \right)}\\ 
			&\hspace{3cm}+ \tfrac{1}{T \underline{\beta}} \sum_{t \in \mc{T}}  \kappa_t \mathrm{e}(t).
		\end{aligned}
	\end{equation}
	\QEDB
\end{theorem}
\begin{proof}
	From Assumption~\ref{ass:recons_error}, by combining the descent lemma \cite[Prop.~A.24]{bertsekas1997nonlinear} and the relation in \eqref{eq:descent} we obtain
	$$
	\theta(\bs{x}^\star_t) \leq \theta(\bs{x}^\star_{t - 1})   - \ell \tfrac{2 - \alpha_t}{2 \alpha_t} \|\Delta_t^\star\|^2 + \tfrac{1 - \alpha_t}{\alpha_t} \|\Delta_t^\star\| \mathrm{e}(t-1)
	$$	
	with probability $1 - \delta$, for all $t \geq \bar{t}$.
	Then, by exploiting the definition of $\beta_t$, which is strictly greater than zero if $\xi_t \in [0, 1/c_t)$, and by focusing on $\ell (2 - \alpha_t)/2 \alpha_t \|\Delta_t^\star\|^2$ and $(1 - \alpha_t)/\alpha_t \|\Delta_t^\star\| \mathrm{e}(t-1)$, we can complete the square by adding and subtracting $((1-\alpha_t)^2/4\alpha^2_t \beta_t) \mathrm{e}^2(t-1)$.
	With standard manipulations one obtains the following expression
	\begin{equation}\label{eq:sample_complexity_structure}
		\begin{aligned}
			\theta(\bs{x}^\star_t) &\leq \theta(\bs{x}^\star_{t - 1}) -\beta_t \left(	\|\Delta_t^\star\| - \tfrac{1-\alpha_t}{2 \alpha_t \beta_t} \mathrm{e}(t-1) \right)^2\\
			&\hspace{4cm}+ \tfrac{(1-\alpha_t)^2}{4 \alpha^2_t \beta_t} \mathrm{e}^2(t-1).
		\end{aligned}
	\end{equation}
	We now fix $T \in \N$ and $\kappa_t = (1-\alpha_t)/2 \alpha_t$, and for any $\bs{x}^\star \in \Theta$, by summing up the inequality above over $t \in \mc{T} \coloneqq \{\bar{t} + 1, \ldots, T + \bar{t}\}$, with probability $1-\delta$ we obtain
	$
	\theta(\bs{x}^\star) \leq \theta(\bs{x}^\star_T) \leq \theta(\bs{x}^\star_{\bar{t}}) - \sum_{t \in \mc{T}} \beta_t (	\|\Delta_t^\star\| -  \tfrac{\kappa_t}{\beta_t} \mathrm{e}(t) )^2 + \sum_{t \in \mc{T}} \tfrac{\kappa^2_t}{\beta_t} \mathrm{e}^2(t).
	$
	Then, by moving $\theta(\bs{x}^\star)$ to the RHS, the term with $\beta_t$ to the LHS, and by introducing $\Delta_{\bar{t}} \coloneqq \theta(\bs{x}^\star_{\bar{t}}) - \theta(\bs{x}^\star)$, a quantity that is always nonnegative, we obtain:
	\begin{equation}\label{eq:1}
		\sum_{t \in \mc{T}} \beta_t \left(	\|\Delta_t^\star\| - \tfrac{\kappa_t}{\beta_t} \mathrm{e}(t) \right)^2 \leq \Delta_{\bar{t}}  + \sum_{t \in \mc{T}} \tfrac{\kappa^2_t}{\beta_t} \mathrm{e}^2(t).
	\end{equation}
	Since $\beta_t >0$ for all $t \in \mc{T}$, the expression in \eqref{eq:1} is amenable to apply the Jensen's inequality \cite[Th.~3.4]{rockafellar1970convex} on the convex function $(\cdot)^2$ to lower bound the summation in the LHS. 
	Before doing that, we normalize both sides by multiplying and dividing by $\bar{\beta} \coloneqq \sum_{t \in \mc{T}} \beta_t$, and we define $\hat{\beta}_t \coloneqq \beta_t/ \bar{\beta}$, thus obtaining 
	$
	\bar{\beta} (\sum_{t \in \mc{T}} \hat{\beta}_t (	\|\Delta_t^\star\| - \tfrac{\kappa_t}{\beta_t} \mathrm{e}(t) )^2 ) \geq
	\bar{\beta} (\sum_{t \in \mc{T}} \hat{\beta}_t	(\|\Delta_t^\star\| - \tfrac{\kappa_t}{\beta_t} \mathrm{e}(t) ) )^2.
	$
	Then, after replacing this latter inequality into \eqref{eq:1}, 
	with few algebraic manipulations we obtain
	$$
	\sum_{t \in \mc{T}} \beta_t \left(\|\Delta_t^\star\| - \tfrac{\kappa_t}{\beta_t} \mathrm{e}(t) \right) \leq \sqrt{\sum_{t \in \mc{T}} \left(\beta_t \Delta_{\bar{t}} + \tfrac{\bar{\beta} \kappa^2_t}{\beta_t} \mathrm{e}^2(t) \right)}.
	$$
	We now bring the term $\sum_{t \in \mc{T}} \kappa_t \mathrm{e}(t)$ in the RHS, and we note that the obtained inequality is still valid  by premultiplying both sides by $1/T$, hence getting the average over the horizon of length $T$.
	Finally, we have that $(1/T) \sum_{t \in \mc{T}} \beta_t \|\Delta_t^\star\| \geq (1/T) \underline{\beta} \sum_{t \in \mc{T}} \|\Delta_t^\star\|$, with $\underline{\beta} \coloneqq \textrm{min}_{t \in \mc{T}} \beta_t$, which directly yields to the relation in \eqref{eq:seq_recons} with probability $1-\delta$.
\end{proof}

%
We can hence upper bound the average value of the residual $\|\Delta_t^\star\|$ over a certain horizon $T$ with the sum of two terms: the first one depends on the initial distance from a global minimum for the unknown potential function $\theta$, while the second one is mainly affected by the reconstruction error~$\mathrm{e}$. 

Remarkably, the terms in the RHS of \eqref{eq:seq_recons} can be made small in two ways: i) by choosing a small step-size $\xi_t$ to make $\kappa_t$ close to zero, or ii) by tuning the product $c_t \xi_t$ close to one, thus producing a large value of $\underline{\beta}$. This latter choice, however, requires an accurate trade-off in tuning the gain $c_t$ and the step-size $\xi_t$, as larger values of $\underline{\beta}$ lead to larger values of the sub-optimal constant $\Delta_{\bar{t}}$.
A possible choice to do not promote aggressive personalized actions establishes that the central coordinator may want to match the lower bound for $c_t$, while striving to counterbalance the choice for $\xi_t$ to possibly obtain a faster convergence rate for Algorithm~\ref{alg:two_layer}.

We now discuss the impact on the bound in \eqref{eq:seq_recons} of the learning strategy $\mathscr{L}$ endowing the central coordinator.
By assuming that $\beta_t$ is fixed over the outer iterations, i.e., $\beta_t = \beta$, we then know from Assumption~\ref{ass:recons_error} that there exists some $\bar{t}$ such that $\mathrm{e}(t) \leq \mathrm{e}(\bar{t})$ for all $t \geq \bar{t}$. Therefore, it follows that the bound in \eqref{eq:seq_recons} satisfies
$$
\tfrac{1}{T} \sum_{t\in \mc{T}} \|	\Delta_t^\star	\| \leq O(1/\sqrt{T}) + O(1). 
$$
As $T$ grows, it is evident that the term $O(1/\sqrt{T})$ vanishes, and $\tfrac{1}{T} \sum_{t\in \mc{T}} \|	\Delta_t^\star	\|$ is confined in a ball whose radius depends on the quantity of agents' feedback made available to perform the reconstruction in the first step of Algorithm~\ref{alg:two_layer} (specifically, on the learning strategy $\mathscr{L}$). In addition, we further stress that Assumption~\ref{ass:recons_error} is quite general and it holds true under mild conditions for \gls{ls} and \gls{gp} approaches to learning $Q \bs{x} + q$:
\begin{itemize}
	\item Note that $Q \bs{x} + q$ can be always modelled as an affine function of the learning parameters $\eta$. Thus, setting up a \gls{ls} approach to minimize the loss between the model parameters $\eta$ and the agents' feedback turns out to be a quadratic program. Due to the large-scale properties of \gls{ls}, the error term $\mathrm{e}(t)$ behaves as a normal distribution, for which Assumption~\ref{ass:recons_error}  holds true (see, e.g., \cite[Lemma~A.4]{notarnicola2020distributed}), and hence $\lim_{t\to\infty} \mathrm{e}(t) = 0$.
	\item In a kernel-based methods, suppose $Q \bs{x} + q$ is a sample path of a \gls{gp} with zero mean and a certain kernel. Due to the large-scale property of such a regressor and under standard assumptions, also in this case Assumption~\ref{ass:recons_error}  holds true (see \cite{simonetto2019personalized}) and $\lim_{t\to\infty} \mathrm{e}(t) = 0$.
\end{itemize}	
We finally remark that, since $\sum_{t \in \mc{T}} \mathrm{e}(t) = o(T)$, for the parametric/non-parametric procedures above
$
\lim_{T \to \infty} \tfrac{1}{T} \sum_{t\in \mc{T}} \|	\Delta_t^\star	\| = 0,
$
which allows us to recover the result in the perfect reconstruction case described in \S \ref{sec:stat_case_1}.

\section{Numerical simulations}
\label{sec:num_sim}

\begin{figure}[!t]
	\centering
	\includegraphics[width=\columnwidth]{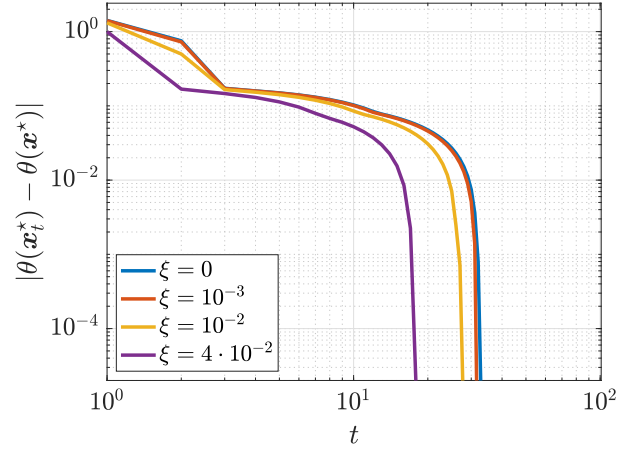}
	\caption{Convergence behaviour of Algorithm~\ref{alg:two_layer} for different values of the step-size $\xi$, with perfect reconstruction of $G(\cdot)$.}
	\label{fig:perf_recon}
\end{figure}

\begin{figure}[!t]
	\centering
	\includegraphics[width=\columnwidth]{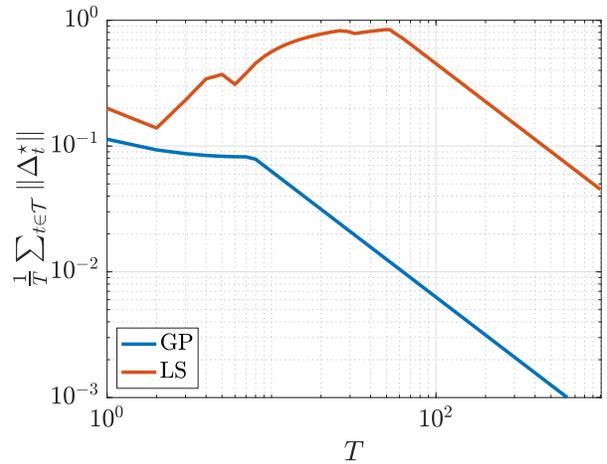}
	\caption{Average value of the residual $\|\Delta_t^\star\|$ over the time horizon $T$, for different learning strategies of the coordinator.}
	\label{fig:path_length}
\end{figure}

\begin{table}
	\caption{Simulation parameters -- Figure~\ref{fig:perf_recon}}
	\label{tab:parameters_perfect_rec}
	\centering
	\begin{tabular}{lll}
		\toprule
		Parameter  & Description   & Value \\
		\midrule
		$\ell$ & Constant of weak convexity & $10.09$\\
		$c_t$ & Personalized functional gain & $20.18$\\
		$P$ & Weight for the extragradient algorithm & $I$ \\
		$l_F$ & Horizontal scale -- \gls{gp} method  & $50$ \\
		$\sigma_F$ & Vertical scale -- \gls{gp} method  & $100$ \\
		\bottomrule
	\end{tabular}
\end{table}

\begin{figure}[!t]
	\centering
	\includegraphics[width=\columnwidth]{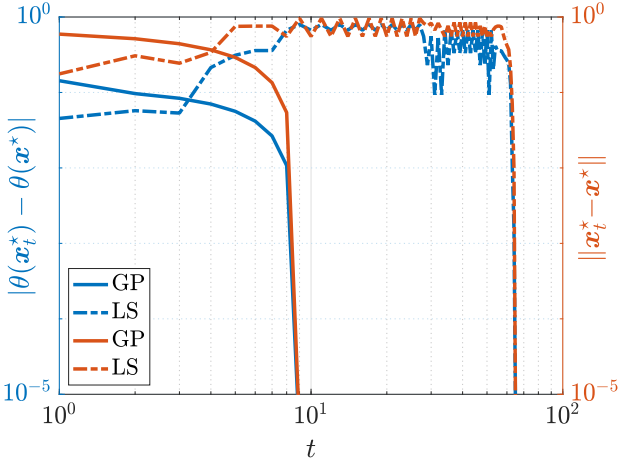}
	\caption{Sub-optimality (left y-axis) and tracking error (right y-axis) of the sequence of \gls{v-GNE} generated by Algorithm~\ref{alg:two_layer} w.r.t. $\bs{x}^\star \in \Theta$, for different learning strategies for the coordinator.}
	\label{fig:eq_error}
\end{figure}

\begin{table}
	\caption{Simulation parameters -- Figure~\ref{fig:path_length} and \ref{fig:eq_error}}
	\label{tab:parameters}
	\centering
	\begin{tabular}{lll}
		\toprule
		Parameter  & Description   & Value \\
		\midrule
		$\ell$ & Constant of weak convexity  & $1.22$\\
		$\mu$ & Mean of the agents' feedback noise $\varepsilon_{i,t}$ & $0$\\
		$\sigma^2$ & Variance of the agents' feedback noise $\varepsilon_{i,t}$ & $25$\\
		$c_t$ & Personalized functional gain  & $2.44$\\
		$\xi_t$ & Personalized functional step-size & $0.41$\\
		$P$ & Weight for the extragradient algorithm  & $I$ \\
		$l_F$ & Horizontal scale -- \gls{gp} method  & $50$ \\
		$\sigma_F$ & Vertical scale -- \gls{gp} method  & $100$ \\
		\bottomrule
	\end{tabular}
\end{table}
 
We verify our findings through numerical instances of the hypomonotone \gls{GNEP} in \eqref{eq:single_prob} with parameters summarized in Tab.~\ref{tab:parameters_perfect_rec}--\ref{tab:parameters}. Both examples involves $N = 20$ agents controlling a scalar variable $n_i = 1$ constrained to the local set $\mc{X}_i = [0,1]$.
Each matrix $A_i$ is constructed so that every pair of consecutive agents are coupled, i.e., $x_i + x_{i+1} \leq b_i$, for all $i \in \mc{I}$, while the shared resource term, $b_i$ is randomly sampled following a uniform distribution on $(0,1)$. In addition, every $Q_i$ and $C_{i,j}$ is randomly generated from a normal distribution, while every $q_i$ follows a uniform distribution on the interval $(-1, 1)$. 

Once given the personalized feedback functional at every outer iteration $t \in \N$, the \gls{v-GNE} is computed by means of the projection-type method described in \cite[Ch.~12]{facchinei2007finite}, thus partially neglecting the multi-agent nature of the problem addressed. Since the algorithm in \cite[Ch.~12]{facchinei2007finite} is tailored for monotone \glspl{VI}, we stress that its convergence is enabled by the incentives in \eqref{eq:pers_feedback}, as $Q + Q^\top$ is not positive semi-definite (i.e., $G(\cdot)$ is not monotone), while $(Q + c_t I) + (Q + c_t I)^\top$ is (i.e., $\bar{G}_t(\cdot)$ is strongly monotone, for $c_t \geq 2 \ell$ -- see Proposition~\ref{prop:strong_conv}).

In Fig.~\ref{fig:perf_recon} is illustrated the convergence boost effect, or error reduction, of the step-size $\xi$ in case of perfect and inexact reconstruction, while in Fig.~\ref{fig:path_length} we compare the behaviour of the metric $(1/T) \sum_{t \in \mc{T}} \|\Delta_t^\star\|$ when the coordinator is endowed with an \gls{ls}-based learning strategy, and a \gls{gp}-based one (covariance matrices computed by using a radial basis function kernel). Since the problem in \eqref{eq:single_prob} makes available the explicit expression of $\theta(\cdot)$ in \eqref{eq:pot_fun_hypo}, we can show the evolution of the potential function evaluated in the sequence of \gls{v-GNE} generated by Algorithm~\ref{alg:two_layer} compared to a global minimum (Fig.~\ref{fig:eq_error} -- left y-axis), as well as the distance between any computed point w.r.t. some point in $\Theta$ (Fig.~\ref{fig:eq_error} -- right y-axis).

\section{Conclusion}
The design of suitable personalized incentives is key to compute \gls{GNE} in quadratic, nonmonotone \gls{GNEP} characterized by bilateral symmetric interactions among agents. The benefit of adopting such functionals are twofold: i) they serve as regularization terms of the agents' cost functions, thus enabling for the practical computation of a \gls{v-GNE} at each outer iteration of the proposed, two-layer procedure, and ii) they provide a mean to achieve faster convergence rate. The proposed algorithm converges to a \gls{GNE} by exploiting the consistency bounds characterizing standard learning procedures for the coordinator, such as \gls{ls} or \gls{gp}.

Future research directions may include extending the proposed approach to time-varying \glspl{GNEP}, also exploring how to avoid to know the constant of weak convexity $\ell$ of the potential function, which represents a fundamental, albeit possibly unknown, parameter to establish convergence.


\balance
\bibliographystyle{IEEEtran}
\bibliography{22_ECC.bib}

\end{document}